\documentclass[11pt]{amsart}

\usepackage{geometry,amssymb,latexsym,xcolor,enumerate}

\title[An overdetermined problem in 2D linearised hydrostatics]{
An overdetermined problem in 2D linearised hydrostatics}
\author{Giovanni Franzina}
\numberwithin{equation}{section}
\newtheorem{theorem}{Theorem}
\newtheorem{proposition}{Proposition}
\newtheorem{lemma}{Lemma}
\newtheorem{cor}{Corollary}

\theoremstyle{definition}
\newtheorem{rmk}{Remark}

\usepackage[
colorlinks=true,urlcolor=cyan,
citecolor=magenta,linkcolor=cyan,
linktocpage,pdfpagelabels,bookmarksnumbered,bookmarksopen,
pagebackref
]{hyperref}

\address[G. Franzina]{\quad\newline
\indent Istituto per le Applicazioni del Calcolo ``M. Picone''
\newline\indent
Consiglio Nazionale delle Ricerche
\newline\indent 
Via dei Taurini, 19, 00185 Roma, Italy}

\subjclass[MSC 2010 Subject Classification]{35P15, 47A75, 49R05, 76D07}
\keywords{Buckling load, Shape optimization problems, Stokes flows, Isoperimetric inequalities}

\begin{document}
\maketitle

\begin{abstract}
In two spatial dimensions, we discuss the relation between
the solvability of Schiffer's overdetermined problem and
the optimisation, among sets of prescribed area, of the first eigenvalue in the buckling problem for a clamped plate
and that of the first eigenvalue of the Stokes operator.
For the latter, we deduce that the minimisers under area constraint
that are smooth and simply connected
must be discs from the fact that
a pressureless velocity is a necessary condition of optimality. 
\end{abstract}

\section{Introduction}
Let $\Omega$ denote a planar open set of finite area with Lipschitz boundary.
We compare the eigenvalues of the buckling problem for a clamped plate
\begin{equation}
\label{1}
	\begin{cases}
	\Delta^2\psi +\lambda \Delta\psi=0\,,&\qquad\text{in $\Omega$,}\\
	\psi = 0\,,&\qquad \text{on $\partial\Omega$,}\\
	\partial_N\psi =0\,,&\qquad \text{on $\partial\Omega$,}
	\end{cases}
\end{equation}
the eigenvalues of interest when considering the steady Stokes equations, i.e.,
\begin{equation}
\label{2}
\begin{cases}
\Delta u +\lambda u = \nabla p\,, &\qquad\text{in $\Omega$,}\\
\nabla\cdot u=0 \,, & \qquad\text{in $\Omega$,}\\
u=0\,, &\qquad\text{on $\partial\Omega$,}
\end{cases}
\end{equation}
and those relative to a third spectral variational problem, namely
\begin{equation}\label{3}
\begin{cases}
\Delta w + \lambda w=g\,,&\qquad\text{in $\Omega$,}\\
\Delta g=0\,,&\qquad\text{in $\Omega$,}\\
\displaystyle\int_\Omega wf \,dx=0 \,,&\qquad\text{for all $f\in C(\overline\Omega)$
with $\Delta f=0$ in $\Omega$.}
\end{cases}
\end{equation}

The sets of all $\lambda>0$ for which
the problems \eqref{1}, \eqref{2}, and \eqref{3} admit a non-trivial solution
are denoted by $\mathfrak{S}^B(\Omega),\mathfrak{S}_p^S(\Omega),
\mathfrak{S}_g^H(\Omega)$, respectively. 

\begin{proposition}\label{prop0}
We have
\[
\mathfrak{S}^B(\Omega)=
\bigcup_{\substack{g\in C(\overline\Omega)\,,\\ \Delta g=0\ \text{in $\Omega$}}}
\mathfrak{S}_g^H(\Omega)
\subseteq\bigcup_{ p \in H^1(\Omega)}
\mathfrak{S}_p^S(\Omega)\,.
\]
If $\Omega$ is simply connected, then the second inclusion holds as an equality.
\end{proposition}

After proving Proposition~\ref{prop0}, we discuss the distinguished role of
 the special value
\begin{equation}
\label{lambda1B}
\lambda_1^B(\Omega) = \min_{\psi\in H_0^2(\Omega)}\frac{
	\displaystyle\int_\Omega(\Delta\psi)^2\,dx}{
	\displaystyle\int_\Omega|\nabla\psi|^2\,dx}\,,
\end{equation}
i.e., the first
eigenvalue in \eqref{1}, in the overdetermined problem
\begin{equation}
\label{5}
\begin{cases}
-\Delta w = \lambda w\,,&\qquad\text{in $\Omega$,}\\
w\equiv{\rm const}\,,&\qquad\text{on $\partial\Omega$,}\\
\partial_N w =0 \,,&\qquad\text{on $\partial\Omega$.}
\end{cases}
\end{equation}
By $\mathfrak{S}^D(\Omega)=\{\lambda_n^D(\Omega)\}_{n\ge1}$
we denote the set of all Dirichlet eigenvalues
of the Laplacian.

\begin{theorem}\label{thm1}
Let $\Omega\subset\mathbb R^2$ be a smooth and bounded open set. Then,  \eqref{5} admits non-trivial solutions if and only if
$\lambda\in\mathfrak{S}^B(\Omega)$
and \eqref{1} has a non-zero solution $\psi$ for which
$\Delta\psi$ is constant along $\partial\Omega$.
In that case, either $\lambda>\lambda_1^B(\Omega)$ 
and $\lambda\in\mathfrak{S}^D(\Omega)$, with $\lambda>
\lambda_2^D(\Omega)$, or $\Omega$ must be a disc.
\end{theorem}

It is well known~\cite{W}
that the boundary of the bounded Lipschitz open sets that support
non-zero solutions of the overdetermined problem \eqref{5} must be analytic. Hence, the smoothness assumption on $\Omega$ is redundant.

As for the last statement of Theorem~\ref{thm1},
it is obtained in~\cite{B} as a consequence of 
Weinstein's inequality, an isoperimetric
property conjectured in~\cite{We} and proved in~\cite{P,F} (see Eq.\ \eqref{weineq} below). Also,
it rephrases another known fact: 
if the minimiser of $\lambda_1^B$ under area constraint, which 
exists~\cite{AB} in the class of simply connected open sets, is smooth then
it must be a disc;
we refer the reader to the related discussion
in~\cite[Chap.~11]{H},
 where 
  a detailed description of this idea, ascribed to N.B. Willms and H.F. Weinberger, is provided; see also~\cite{KP}.
  We just recall an expedient fact in both applications, that
if $\Omega$ is a smooth open set which is not a disc,
then
by acting with infinitesimal rotations 
on solutions of the overdetermined problem \eqref{5}
one may produce
Dirichlet eigenfunctions relative to $\lambda$ with three nodal domains at least; in view of Faber-Krahn inequality, of the equality case in the
universal Ashbaugh-Benguria inequality, and of that in Weinstein's inequality, that would yield a contradiction unless $\lambda>\lambda_1^B(\Omega)$.
\vskip.2cm

We will prove  the rest of Theorem~\ref{thm1} in Section~\ref{sec:2}. 

\begin{rmk}\label{inci}
Incidentally, we recall that $\Delta\psi\big\vert_{\partial\Omega}\equiv {\rm const}$ is a necessary condition of optimality for the minimisation of $\lambda_1^B(\Omega)$ among sets of given area, assuming simplicity. This implies a restriction~\cite{BP}.
This condition is relevant also to merely critical shapes, see~\cite{BL,BL2}.
\end{rmk}

\begin{rmk}\label{orthoharm}
Any solution $w$ of the overdetermined problem \eqref{5}
must be orthogonal in $L^2(\Omega)$ to harmonic functions in $\Omega$.
Indeed, if $\Delta h=0$ then
\[
\lambda\int_\Omega w\,h\,dx = -\int_\Omega h\Delta w\,dx
=\int_\Omega \nabla h\cdot\nabla w\,dx = \int_{\partial\Omega} w\partial_Nh\,d\mathcal{H}^1= {\rm const}\,{\cdot}\!\!\int_\Omega\Delta h\,dx=0\,,
\]
where we also used the three equations in \eqref{5} for the first equality, for the second one, and for the fourth one, respectively. Thus, \eqref{5} can be seen as the pairing of problem \eqref{3}, with $g=0$, and of an additional Dirichlet boundary condition.
\end{rmk}

Given $\nu\in(0,+\infty)$, the conditions in Theorem~\ref{thm1} relate to the geometric rigidity 
of some special cellular flows 
confined within contractible rigid walls
solving the 2D Navier-Stokes equations
\begin{equation}
\label{4}
\begin{cases}
\partial_t v + (v\cdot\nabla)v = \nu\Delta v-\nabla p \,,&\qquad\text{in $\Omega$,}\\
\nabla\cdot v=0\,,&\qquad\text{in $\Omega$,}\\
v=0\,,&\qquad\text{on $\partial\Omega$.}
\end{cases}
\end{equation}
The pressure gradient in \eqref{4} is often thought of as a Lagrange multiplier
arising with the incompressibility constraint $\nabla\cdot v=0$, like in the case of steady Stokes equations. When solving for the pressure function,
it is immediate to recognise that it
is a harmonic function. Since it is defined up to additive constant,
one might expect the equations for 
$p$ to be automatically supplemented with boundary conditions that only involve $\nabla p$, such as Neumann conditions, which may look natural at a first glance.
Notwithstanding, such requirements for the pressure are extremely rigid, as one can imply from the following statement. 

\begin{proposition}\label{prop1}Let $\Omega\subset \mathbb R^2$ be a smooth and simply connected open set and let $\lambda>0$.
A necessary and sufficient condition that either
condition  (and hence both)  in Theorem~\ref{thm1} be valid
is that 
there exist a vector field $u$  for which \eqref{2} holds with pressure $p$ satisfying 
\begin{equation}
\label{neumann}
	\partial_N p=0\,,\qquad\text{on $\partial\Omega$.}
\end{equation}
In fact, if \eqref{2} holds, then
the Neumann boundary condition \eqref{neumann}
is equivalent to condition
\begin{equation}
\label{nonlocal}
p \equiv {\rm const} \,,\qquad\text{in $\Omega$.}
\end{equation}
Eventually, this happens if and only if for any $\nu\in(0,+\infty)$  the function
\begin{equation}
\label{cell}
v(x,t)= e^{-\nu\lambda t}u(x)\,,
\end{equation}
is a solution of \eqref{4} with  $\nabla p=0$.
\end{proposition}
Note that in Proposition~\ref{prop1} the primitive variable at hand
is a velocity field $u$ solving \eqref{2}.
The assumption that $\Omega$ is smooth assures that $u$ is smooth all the way up to the boundary, by Solonnikov's estimates~\cite{S1,S2}. The assumption that $\Omega$ is simply connected, in turn, assures that we can write
$u=\nabla^\perp \psi$ for a scalar function and then the boundary equations $\psi=
\partial_N\psi=0$ are inherited from the no-slip condition $u=0$ on $\partial\Omega$.
\medskip

Then, we consider the least eigenvalue
for the problem \eqref{2}, viz.
\begin{equation}
\label{lambda1S}
\lambda_1^S(\Omega) = \min_{u\in H^1_0(\Omega\mathbin{;}\mathbb R^2)}
\left\{\frac{\displaystyle\int_\Omega |\nabla u|^2\,dx}{\displaystyle
\int_\Omega|u|^2\,dx}\mathbin{\colon} \text{$\nabla\cdot u=0$ in $\Omega$}
\right\}\,,
\end{equation}
which
was considered in the recent paper~\cite{HMP}
from the perspective of spectral optimisation.
As an immediate
consequence of Lemma~\ref{propgengenopen} and Lemma~\ref{propSCs}, proved in Section~\ref{sec:3}, we can prove the following fact. The equality
case is already pointed out in~\cite[Sect.~1.3]{HMP}.
 \begin{proposition}\label{propSC}
\(
\lambda_1^B(\Omega) 
\ge\lambda_1^{S}(\Omega)
\),
with equality if $\Omega\subset\mathbb R^2$ is simply connected.
\end{proposition}
 
We recall {\em Weinstein's inequality} which gives a lower bound
for the first eigenvalue $\lambda_1^B$ in \eqref{1} in terms of the second Dirichlet eigenvalue $\lambda_2^D$ of the Laplacian: for all planar open sets $\Omega$ one has
\begin{equation}
\label{weineq}
	\lambda_1^B(\Omega)\ge\lambda_2^D(\Omega)\,,
\end{equation}
with equality if and only if $\Omega=\Omega_\star$,
where $\Omega_\star$ is the disc contouring as much area as $\Omega$.
As a consequence of the inequality in Proposition~\ref{propSC}
and of the equality case in Weinstein's inequality, the isoperimetric property of the disc for $\lambda_1^S$, were it valid, would confirm
the P\'olya-Szeg\"o conjecture for $\lambda_1^B$, i.e., that
{\em the disc minimises the least buckling eigenvalue among open sets of given area}. 
The local optimality of the disc proved in~\cite[Theorem~2]{HMP} is consistent
with this long standing conjecture.

\begin{rmk}[Optimality conditions for $\lambda_1^S(\Omega)$]\label{necopt}
If $\Omega$ is a bounded and smooth simply connected open set
for which $\lambda_1^S(\Omega)$ simple and minimal among open sets of given area
then the corresponding eigenfunctions $u$ solve \eqref{2}
with
a pressure for which \eqref{nonlocal} (or, equivalently, with \eqref{neumann}) holds.
\end{rmk}

The necessary condition of optimality indicated in Remark~\ref{necopt} can be deduced
from Theorem~\ref{thm1} and from the equality case of Proposition~\ref{propSC}
on simply connected domains (see also Remark~\ref{inci}). Since
the least Stokes eigenvalue is simple on smooth open sets~\cite[Theorem~4]{HMP}, we have the following.

\begin{cor}\label{prop2d}
Let $\Omega$ be a minimiser for $\lambda_1^S(\Omega)$ under
area constraint. If $\Omega$ is smooth, bounded, and simply connected,
then it must be a disc.
\end{cor}

A quasi-open set that minimises 
$\lambda_1^S$ under area constraint
is known to exist~\cite[Theorem~1]{HMP}. Yet, the difficult regularity issue is still open, nonetheless.
Note the difference with the three-dimensional case, in which case the ball is not a minimiser~\cite[Theorem~3]{HMP}.
\medskip

Here is a final comment on the rigidity of requirement \eqref{neumann} for the pressure in Stokes equations \eqref{2}.

\begin{cor}
If $\Omega\subset \mathbb R^2$ is a smooth, bounded, and simply connected open set and \eqref{2} with $\lambda=\lambda_1^S(\Omega)$ admits solutions for which the pressure satisfies the Neumann boundary conditions \eqref{neumann}, then $\Omega$ must be a disc.
\end{cor}

It is not difficult to see that \eqref{neumann} may hold if $\Omega$ describes an infinite straight channel. Hence, the assumption that  $\Omega$ be bounded cannot be removed.

\section{Overdetermined problems}\label{sec:2}
We will make repeatedly use of the following Lemma. We refer to~\cite{QV}
for a more general statement in the case of smooth open sets.

\begin{lemma}\label{lemma-orto}
Let $\Omega$ be Lipschitz and $\psi\in H^2(\Omega)\cap H^1_0(\Omega)$. Then $\psi\in H^2_0(\Omega)$ if and only if
\begin{equation}
\label{orthoha}
\int_\Omega f \Delta \psi\,dx =0\,,\qquad\text{for all $f\in C(\overline\Omega)$
with $\Delta f=0$ in $\Omega$.}
\end{equation}
\end{lemma}
\begin{proof}
Integrating by parts twice yields \eqref{orthoha} for $\psi\in C^\infty_0(\Omega)$, and hence for all $\psi\in H^2_0(\Omega)$ by
a density argument. Conversely, let $\psi \in H^2(\Omega)\cap H_0^1(\Omega)$
be orthogonal to all harmonic functions. Then, 
the vector field $\nabla\psi$ has divergence in $L^2(\Omega)$ and its normal trace at the boundary is a well defined element $\partial_N\psi$ of $L^2(\partial\Omega)$.
If $f\in C(\overline\Omega)$ is harmonic,
\[
\int_{\partial\Omega}f\partial_N\psi\,d\mathcal{H}^1 = \int_\Omega \nabla f \cdot \nabla \psi\,dx = 0\,,
\]
where we also used, in the first integration by parts, that $\Delta\psi$
is orthogonal to $f$ and, in the second one, that the boundary
trace of $\psi $ is zero. By the solvability of the Dirichlet problem for the Laplace equation for all boundary values in the trace space $H^{1/2}(\partial\Omega)$,
we deduce that the boundary trace $\partial_N\psi$ 
is orthogonal to all elements of a dense subset
of $L^2(\partial\Omega)$ and hence must be zero.
\end{proof}

\subsection{Proof of Proposition~\ref{prop0}}We split the proof into steps.
\subsubsection{$\mathfrak{S}^B(\Omega)\subseteq \mathfrak{S}_0^H(\Omega)$.}
Let $\lambda\in\mathfrak{S}^B(\Omega)$
and let $\psi$ be a solution of \eqref{1}. By Lemma~\ref{lemma-orto},
$w = \Delta \psi$ defines an element of $H^1(\Omega)$ orthogonal to all harmonic functions and clearly
we have $\Delta w +\lambda w=0$. Thus, \eqref{3} holds with $g=0$.

\subsubsection{$\mathfrak{S}_g^H(\Omega)\subseteq\mathfrak{S}^B(\Omega)$
for all harmonic $g$}
Let $g$ be harmonic, let $\lambda\in\mathfrak{S}_g^H(\Omega)$, and let $w$ solve \eqref{3}. We let $\psi\in H^1_0(\Omega)$ be the solution of Poisson's equation
\begin{equation}
\label{poissong}
\begin{cases}
	\Delta \psi = \lambda^{-1}({\displaystyle
	g-\Delta w})\,,&\qquad\text{in $\Omega$,}\\
	\psi=0\,,&\qquad\text{on $\partial\Omega$.}
\end{cases}
\end{equation}
Then, by~\cite[Theorem~8.12]{GT}, one has in fact $\psi \in H^2(\Omega)$. By inserting the equation for $w$ in \eqref{poissong} we also see that $w=\Delta \psi$. Hence, by Lemma~\ref{lemma-orto},
the orthogonality of $w$ against harmonic functions implies
that $\psi\in H_0^2(\Omega)$, too. Therefore, $\psi$ solves \eqref{1}.

\subsubsection{$\mathfrak{S}^B(\Omega)\subseteq \mathfrak{S}_p^S(\Omega)$
for an appropriate $p$}
Let $\psi$ be a solution of \eqref{1}. By Lemma~\ref{lemma-orto}, $w=\Delta\psi$ must be orthogonal in $L^2(\Omega)$ to all harmonic functions in $\Omega$.
Also,
the equation $\Delta w+\lambda w=0$ holds. Then, we notice that
\begin{equation}
\label{extharm2}
\text{$h:=\Delta\psi + \lambda\psi $ is the harmonic extension
of the boundary trace of $\Delta\psi$.}
\end{equation}
Hence, $u=\nabla^\perp\psi$ is
a solution of \eqref{2} with $\nabla p = \nabla^\perp h$. Equivalently,
$h$ is the harmonic conjugate of $p$.
This means that $p$ is defined
up to an additive constant as the harmonic conjugate of $h$, i.e.,
the complex function whose real and imaginary parts are
$p$ and $ h$, respectively, is holomorphic.

\subsubsection{The last statement}
We assume that $\Omega$ is simply connected.
Then, we can find $\psi$
with $u=\nabla^\perp\psi$
where $u$ is a given solution of \eqref{2}, with $p\in H^1(\Omega)$
and
 $\lambda\in\mathfrak{S}_p^S(\Omega)$. We prove that 
$\lambda\in\mathfrak{S}_0^H(\Omega)$.
Indeed, $\psi\in H^2_0(\Omega)$ by \eqref{2} and so $w=\partial_{x_1}u^2-\partial_{x_2}u^1$ 
is orthogonal to harmonic functions by Lemma~\ref{lemma-orto}.
As $w=\Delta\psi$, after taking the curl of the first equation in \eqref{2}, 
we arrive at
\eqref{3} with $g=0$.
\qed\\

Henceforth in this section,
we will sometimes understand the solutions in the classical sense. In the case
of a smooth domain $\Omega$,
the
 global
estimates due to Solonnikov~\cite{S1,S2,L} for the Dirichlet problem in linearised hydrostatics \eqref{2} imply
the regularity up to the boundary of $u=\nabla^\perp\psi$
whenever
$\psi\in H^2_0(\Omega)$ is a weak solution
of \eqref{1}, and hence $w=\partial_{x_1} u^2-\partial_{x_1}u^1$ is also smooth up to the boundary; then, so
is $\psi$ because of the regularity of the forcing term in Poisson's equation
$\Delta\psi=w$, with $\psi=0$ on $\partial\Omega$.

\subsection{Proof of Theorem~\ref{thm1}}
For the last statement, we refer to the proof of~\cite[Theorem~11.3.7]{H}, where the idea is credited to Weinberger and Wills. 

If $\lambda>0$ is such that \eqref{5} has a non-trivial solution $w$ which takes
the constant value $c$ on the boundary, then $\psi=w-c$ satisfies the boundary conditions in \eqref{1} and $\Delta\psi+\lambda \psi $ is constant, hence harmonic. Also, in view of Remark~\ref{orthoharm}, we have
$\lambda\in \mathfrak{S}_0^H(\Omega)$ and that is contained in $\mathfrak{S}^B(\Omega)$ by Proposition~\ref{prop0}.

Conversely, if $\lambda\in\mathfrak{S}^B(\Omega)$ and $\psi$ solves \eqref{1}, then $w:=\Delta\psi$ satisfies the first equation in \eqref{5}. If $w$ is constant on $\partial\Omega$, then
the harmonic exstension $h=w+\lambda \psi$ of the boundary values of $w$ must be constant as well by the maximum principle,
and so will be its harmonic conjugate $p$. This proves $u:=\nabla^\perp\psi$ to solve \eqref{2} with
 constant pressure. Therefore,  both the tangential and the 
 normal component of $\Delta u$ vanish at the boundary. As $\nabla^\perp w=-\Delta u$, we have $\nabla w=0$ along $\partial\Omega$, and that
 gives the last two equations in \eqref{5}.\qed

\subsection{Proof of Proposition~\ref{prop1}}
Let $u $ solve \eqref{2} for an appropriate pressure function $p$. Recall that both $u$ and $p$ are smooth functions up to the boundary.
Since we are assuming $\Omega$ to be simply connected, there exists $\psi$ with $u=\nabla^\perp\psi$. 

The harmonic function $h=\Delta\psi + \lambda\psi$
is constant if, and only if, $\nabla p = \nabla^\perp h$ is
identically zero, because $\Omega$ is connected. Also, $h$ is the harmonic function with the same boundary
values as $\Delta\psi$, hence $h$ is constant in $\Omega$
if and only if $\Delta\psi$ is constant on $\partial\Omega$.

Conversely, the homogeneous Neumann boundary conditions \eqref{neumann} for the pressure imply that
$\nabla^\perp (\Delta\psi )\cdot N = \Delta u \cdot N = 0$ along $\partial\Omega$, whence it follows that $\Delta\psi$ is constant on $\partial\Omega$, because
the boundary of a simply connected {\em planar} set is connected. 
%
%

We set $w=\Delta\psi$. Recalling that
$\nabla^\perp w=-\Delta u$, in view of \eqref{2} we have $(u\cdot\nabla)w=(\nabla p-\lambda u)\cdot\nabla \psi$. Hence, recalling
that $u=\nabla^\perp\psi$, we see that $\nabla p =0$ implies that the convective
term disappears from \eqref{4} if $v$ is defined by \eqref{cell}, and the latter
obviously defines a solution of the heat equation. Thus, we have seen that
solutions of \eqref{2} with constant pressure make \eqref{cell} solve \eqref{4}.

Conversely, let $v$ be defined by
\eqref{cell} and let \eqref{4} be valid. As
$\partial_t v = \nu \Delta v$, the convective term $(v\cdot\nabla)v$ must clear off,
which happens only if $\nabla p$ and $\nabla \psi$ are orthogonal, where $u=\nabla^\perp\psi$. In view of the condition $u=0$ on 
$\partial\Omega$, that implies $ \nabla p =0$ on $\partial\Omega$.
\qed

\section{Relations between the least eigenvalues} \label{sec:3}

\begin{lemma}\label{byparts}Let $\Omega$ be a planar open set of finite area.
Let $\psi\in H^2_0(\Omega)$ and
 $u=\nabla^\perp \psi $. Then $u\in H^1_0(\Omega;\mathbb R^2)$,
with $\nabla\cdot u=0 $, and 
\[
\frac{\displaystyle\int_\Omega|\nabla u|^2\,dx}{\displaystyle\int_\Omega|u|^2\,dx}
=\frac{\displaystyle\int_\Omega(\Delta\psi)^2\,dx}{\displaystyle\int_\Omega|\nabla\psi|^2\,dx}
\]
\end{lemma}
\begin{proof}
One uses that $\partial_{xy}^2\psi-\partial_{xx}^2\psi\partial_{yy}^2\psi$ 
integrates to zero in $\Omega$.
\end{proof}

\begin{lemma}\label{propgengenopen}
Let $\Omega$ be a planar open set. Then
\[
\lambda_1^B(\Omega)\ge \lambda_1^S(\Omega)\,.
\]
\end{lemma}
\begin{proof}
Given $\varepsilon>0$,
we can find $\psi\in C^\infty_0(\Omega)$ with
\[
 \lambda_1^B(\Omega)+\varepsilon\ge \frac{\displaystyle\int_\Omega(\Delta\psi)^2\,dx}{\displaystyle\int_\Omega|\nabla\psi|^2\,dx}\,.
 \]
Note that
 $u=\nabla^\perp \psi$
 is  admissible for the definition of $\lambda_1^S(\Omega)$. Then, the conclusion follows by Lemma~\ref{byparts}.
 \end{proof}

\begin{lemma}\label{propSCs}
Let $\Omega$ be simply connected. Then 
\[
\lambda_1^{S}(\Omega)\ge\lambda_1^B(\Omega)\,.
\]
\end{lemma}
\begin{proof}
Let $\varepsilon>0$.
Recall that $C^\infty_0(\Omega;\mathbb R^2) $ is dense in $H^1_0(\Omega;\mathbb R^3)$.  Then, by arguing as done in Lemma~\ref{propgengenopen},
we find $u\in C^\infty_0(\Omega;\mathbb R^2)$, with $\nabla\cdot u=0$, 
\[
 \lambda_1^S(\Omega)+\varepsilon\ge \frac{\displaystyle\int_\Omega|\nabla u|^2\,dx}{\displaystyle\int_\Omega|u|^2\,dx}\,.
\]
By assumption,
there exists $\psi\in H^2_0(\Omega)$ with $\nabla^\perp \psi = u$
and we conclude by Lemma~\ref{byparts}.\end{proof}

We end this section by observing some direct consequences
of Theorem~\ref{thm1} concerning a variational
description of the buckling eigenvalues $\lambda$, with eigenfunction $\psi$, in terms of the laplacian $w = \Delta \psi$.

\begin{rmk}
If $\psi$  solves \eqref{1}, then the function
$w =\Delta\psi$ and the harmonic extension
\[
h=w+\lambda \psi
\]
of the boundary values of $w$ are smooth up to the boundary and satisfy the identity
\begin{equation}
\label{nice-identity}
\int_\Omega |\nabla h|^2\,dx = \int_{\partial\Omega} w\partial_N w\,d\mathcal{H}^1 = \int_\Omega |\nabla w|^2\,dx - \lambda \int_\Omega w^2\,dx\,.
\end{equation}
Hence, whenever 
we start from a buckling eigenfunction $\psi$ with eigenvalue
$\lambda$ we arrive at an eigenpair  $(w,\lambda)$  for the eigenvalue
problem \eqref{3} with
\begin{equation}
\label{conseq}
\lambda = \frac{\displaystyle\int_\Omega |\nabla w|^2\,dx-\int_{\partial\Omega}w\partial_Nw\,d\mathcal{H}^1}{\displaystyle\int_\Omega w^2\,dx}
\le\frac{\displaystyle\int_\Omega |\nabla w|^2\,dx}{\displaystyle\int_\Omega w^2\,dx}\,,
\end{equation}
with equality if, and only if, $\partial_Nw =0$ identically on $\partial\Omega$, i.e.,
$\Delta\psi\equiv c$ on $\partial\Omega$. 
In the equality case of \eqref{conseq}, we have $\lambda\ge \lambda_2^D(\Omega)$. 
In fact, whether or not the equality holds,
it is known that either $\lambda=\lambda_1^B(\Omega)$
and $\Omega$ is a disc, so that $\lambda_1^B(\Omega) = \lambda_2^D(\Omega)$, or else $\lambda$
is also a Dirichlet eigenvalue
of the Laplacian, with $\lambda> \lambda_2^D(\Omega)$. 
\end{rmk}

\end{document}